\documentclass[leqno]{article} 
\usepackage{amsmath,amsthm}
\usepackage{amssymb,amscd,latexsym}
\usepackage{boxedminipage}
\usepackage{eepic}
\usepackage{epic}
\usepackage{wrapfig}
\usepackage{graphicx}

\theoremstyle{plain}

\newtheorem*{theo}{Theorem}
\newtheorem*{def-theo}{Definition-Theorem}
\newtheorem*{fact}{Fact}

\theoremstyle{definition}
\newtheorem*{example}{Example}
\newtheorem*{definition}{Definition}

\theoremstyle{remark}

\newtheorem{thm}{Theorem}[section]
\newtheorem{cor}[thm]{Corollary}

\theoremstyle{definition}

\newtheorem{rem}[thm]{Remark}

\newtheorem{ass}[thm]{Assertion}
\newtheorem{prob}[thm]{Problem}
\newtheorem{conj}[thm]{Conjecture}

\numberwithin{equation}{section}
\newcommand{\bp}{\begin{pmatrix}}
\newcommand{\ep}{\end{pmatrix}}
\newcommand{\bps}{\begin{smallmatrix}}
\newcommand{\eps}{\end{smallmatrix}}
\def\C{{\mathbb C}}

\def\R{{\mathbb R}}
\def\Z{{\mathbb Z}}

\def \0{{\bf 0}}
\def \1{{\bf 1}}

\def \mf#1#2#3#4{
\xymatrix{{#1}\  \ar@<0.4ex>[r]^{{#2}} & \ {#4}
\ar@<0.4ex>[l]^{{#3}}
}
}

\def \mfs#1#2#3#4{\!
\xymatrix@C=1.5em{{#1} \! \ar@<0.2ex>[r]^{{#2}} & \! {#4}
\ar@<0.2ex>[l]^{{#3}}
}
\!}

\def \mfl#1#2#3#4{
\xymatrix@C=2.6em{{#1}\  \ar@<0.4ex>[r]^{{#2}} &\  {#4}
\ar@<0.2ex>[l]^{{#3}}
}
}

\def \mfss#1#2#3#4{\!
\xymatrix@C=1.5em{{#1} \ar@<0.3ex>[r]^{{#2}} & {#4}
\ar@<0.3ex>[l]^{{#3}}
}
\!}

\begin{document}

\title{{\Large Inversion formula 
for 
the growth function\\
 of a cancellative  monoid
}}

\author{Kyoji Saito}
\date{}
\maketitle

 {\renewcommand{\baselinestretch}{0.1}

\begin{abstract}
\noindent
Let $(M,\deg)$ be a cancellative monoid $M$ equipped with a discrete degree map $\deg\!:\!M\! \to\!\R_{\ge0}$, and let $P_{M, \deg}(t)\!:=\!\!\sum_{u\!\in\! M/\!\sim}\!t^{\deg(u)}$ be its generating series
 (here $u\!\sim\! v \underset{def}{\Leftrightarrow} u|_lv \ \&\ v|_lu$, {\footnotesize \S2 Def.}). We prove the {\it inversion formula}
\vspace{-0.05cm}
 \[
P_{M,\deg}(t)\cdot N_{M,\deg}(t)=1
\] 
where the second factor in LHS is a suitably signed generating series
\vspace{-0.05cm}
\[
\begin{array}{l}
N_{M,\deg}(t):= 1+\sum_{T\in \mathrm{Tmcm}(M,I_0)}(-1)^{\#J_1+\cdots+\#J_{n}-n+1} \sum_{\Delta\in |T|} t^{\deg(\Delta)}
 \vspace{-0.1cm}
\end{array}
\]
of the union of some minimal common multiple sets $|T|$, where $T$s, called  towers of minimal comon multiple sets in $M/\!\!\sim$,  form a tree $\mathrm{Tmcm}(M,I_0)$. 

If the  monoid  is $(M,\deg)\!=\!(\Z_{>0},\log)$,  we get  Riemann's zeta function $P_{\Z_{\!>\!0},\log}(\exp(-s))\!=\!\zeta(s)$ as the growth function. Then the inversion formula turns out to be the Euler product formula $\zeta(s)\prod_{p\in I_0}(1\!-\!p^{-s})\!=\!1$. 
\end{abstract}

\tableofcontents
\footnote
{{\bf \small Acknowledgement:} The author is deeply grateful to Tadashi Ishibe for various discussions on cancellativity of monoids as well as  several explicit calculations of examples of the skew growth functions of monoids, which inspired and supported the present work.
He is also grateful to Scott Carnahan for the careful reading of the manuscript.}

 \section{Introduction} 
Let $M$ be a monoid, i.e.\ a semigroup with the unit 1, and let  
$\deg: M\!\to\! \R_{\ge0}$ 
be a discretely valued degree map on $M$ (see \S4). Then the (spherical) 
growth function of $M$ with respect to $deg$ is defined as the generating series:
\vspace{-0.05cm}
\[
\vspace{-0.05cm}
P_{M,\deg}(t):=\sum_{u\in M/\!\sim}t^{\deg(u)},
\vspace{-0.2cm}
\]
where $u\sim v$ for $u,v\in M$ means $u\mid_l v$ and $v\mid_l u$ (see \S2 Def.).
Even though the definition of the growth function of a monoid looks a simple generalization of that for a group,  not much work seems to be available except for some special case studies and some general works on language ([A-N][B][C][C-F][D][G][G-P][I1][S1234][S-I]), and we know little about its general nature. The purpose of the present paper is to give a new approach to the growth function of a monoid by giving a presentation of its inversion function  $\frac{1}{P_{M,I}(t)}$ by a certain {\it "skew generating function"\footnote
{By a {\it skew generating function}, we mean a suitably term-wise signed generating series.
}  of some common multiple sets in the monoid $M$}.

Let us explain the idea of the inversion function  in the most naive case studied in [A-N] and [S23].  Let  $M$ be a 
monoid generated by a finite set $I$ with positive homogeneous relations. It admits naturally an integral degree map by giving weight 1 to each generator in $I$. Suppose, further,  that $M$ is  cancellative  and that any subset $J$ of  $I$  admits either the least right common multiple $\Delta_J$ or no common multiple in $M$ (typically the case for Artin monoids [B-S][D]). Then the inversion function  $P_{M,I}(t)^{-1}$  
is given by 
\vspace{-0.1cm}
\[\vspace{-0.1cm}
N_{M,I}(t):={\sum}_{J\subset I}(-1)^{\#J}t^{\deg(\Delta_J)}
\vspace{-0.05cm}
\]
where the summation index $J$ runs over all subsets of $I$ whose least right  common multiple exists. \footnote
{In this case, $N_{M,I}(t)$ is a polynomial. Actually, it is shown that the coefficients of $N_{M,I}(t)$ give the recursion relation on the sequence of coefficients of $P_{M,\deg}(t)$ (which is generalized in \S5 of the present paper). Zero loci of the polynomial $N_{M,I}(t)$ plays a quite important role in the study of limit functions in [S1234][K-T-S], which motivated the present work.
}

That is, {\it the inversion function 
of the growth function for this class of monoids  
is given by  the skew generating function $N_{M,I}(t)$ of the set 
$\{\Delta_J\}_{J\subset I}$ of all least common right multiples for subsets $J$ of $I$}. However, in general, a monoid may not admit the least right common multiple $\Delta_J$ for a given finite subset $J$ of $M$.  More precisely, even if there exist some common right-multiples  of  $J$, there may not exist the unique, up to units, least (with respect to the partial order induced from left-division relation) element among them.
Thus, the lack of the least common multiples in $M$ may look to be an obstruction to generalize the above  inversion formula to wider class of monoids.

The purpose of the present paper is to resolve this problem as follows.  By assuming the descending chain condition on $M/\!\!\sim$ with respect to the partial order induced by left-divisibility relation, for any given finite set $J\!\subset \! M/\!\!\sim$, 
instead of considering the least common multiple $\Delta_J$, 
we are able to consider the set $\mathrm{mcm}(J)$ of {\it minimal common right-multiples} of $J$.
However, still the datum $\{\mathrm{mcm}(J)\}_{J\subset I}$ is not sufficient to recover the inversion formula, since in general (as we shall see in examples), a subset $J'$ of $\mathrm{mcm}(J)$ may have common right-multiples. So we need to consider the set $\mathrm{mcm}(J')$ 
for a subset $J'$ of $\mathrm{mcm}(J)$. Then again, we may need to consider $\mathrm{mcm}(J'')$ for a subset $J''\!\subset\!\mathrm{mcm}(J')$, and so on. Repeating this process, we are necessarily lead to consider a {\it tower} $T$: a finite sequence  $I\!\supset\! J_1, J_2,\cdots\!,J_n$ of subsets of $M/\!\!\sim$ such that $J_2\!\subset\!\! \mathrm{mcm}(J_1), \!\cdots\!, J_n\!\subset\!\!\mathrm{mcm}(J_{n-1})$}.\!\footnote
{For a technical reason, we assume further $\#J_k>1$ and $\mathrm{mcm}(J_k)\not=\emptyset$ for $1\le k\le n$,
}\!
It is convenient to put an {\it oriented graph\! structure} on the set $\mathrm{Tmcm}(M)$ of all towers, by putting arrows from a tower to its immediate successors.  The graph decomposes into a disjoint union of rooted 
trees $\mathrm{Tmcm}(M)\!=\!\sqcup_{I\subset M}\mathrm{Tmcm}(M,I)$ where the label $I$, as the ground of towers, runs over all subsets of $M$ satisfying the minimality condition $I\!=\!\min(I)$. 

 If $M$ has a discrete degree map $\deg\!:M\!\to\! \R_{\ge0}$,\footnote{
 A degree map is a monoid homomorphism $deg: M\!\to\!\R_{\ge0}$. 
 Its range may not necessarily be contained in an arithmetic progression, but we assume only its discreteness (see \S4 Definition). Therefore, the (skew-)growth functions are not  necessarily power series in the usual sense, but they may better be regarded as Dirichlet series (see Remark 4.3 and \S6 Example).
 \vspace{-0.6cm}
 }
 we can define not only the growth function $P_{M,\deg}(t)$ as before, but also the {\it skew generating function} 
\[
N_{M,I}(t):=1+  \sum_{T\in \mathrm{Tmcm}(M,I)}\!(-1)^{\#J_1+\cdots+\#J_{n}-n+1} \ \sum_{\Delta\in |T|}\ t^{\deg(\Delta)}
\]
with respect to $\deg$ for the tree $\mathrm{Tmcm}(M,I)$\!  of all towers grounded over  the label set\! $I$.\ Here,\! the summation index $T$ runs over the vertex of  the tree $\mathrm{Tmcm}(M,I)$, $J_1,\cdots,J_{n}$ are the stages of the tower $T$ and 
$|T|:=\mathrm{mcm}(J_n)$ is the set of the minimal common multiples on the top stage of the tower $T$ (see \S4).  

In particular, for the label set $I_0\!=\!\min(M/\!\!\!\sim\setminus\{[1]\})$ (the  set of minimal elements of $M$ with respect to the partial ordering induced by the left  division relation), we  set $N_{M,\deg}(t):=N_{M,I_0}(t)$. Then, as the main result of the present paper, we obtain the inversion formula (\S5 Theorem):
\vspace{0.1cm}
\[
P_{M,\deg}(t)\cdot N_{M,\deg}(t)=1.
\vspace{0.1cm}
\]

 We stress  that {\it the first factor $P_{M,\deg}(t)$ in the formula describes the growth nature of the monoid $M$ and the second factor $N_{M,\deg}(t)$ describes the multiplicative nature of the monoid $M$, which are combined as in the formula. }
In order to illustrate this nature of the formula, let us consider the monoid $M\!=\!\Z_{>0}$ with the ordinary product  structure and take "$\log$" to be the degree map (\S5 Example 2). Then, by a change of variable $t=\exp(-s)$,  we have $P_{\Z_{>0},\log}(\exp(-s))=\zeta(s)$ (Riemann's zeta-function) and $N_{M,\deg}(\exp(-s))=\prod_{\substack{p: \text{primes}}}(1-p^{-s})$  so that the inversion formula turns out to be the  well known Euler product formula:
\vspace{-0.1cm}
\[
\zeta(s)\ \cdot\! \prod_{\substack{p: \text{prime}\\\text{numbers}}}(1-p^{-s})\ =\ 1.
\vspace{-0.2cm}
\]

The construction of the paper is as follows. 
In \S2, we fix basic concepts and notation on division theory on a monoid assuming a descending chain condition, where we introduce two operations "cm" (common multiple), "min" (minimal) and their composition "mcm=min$\cdot$cm" (minimal common multiple) on subsets of a monoid. The concept of a tower of minimal common multiples, and the graph structure on the set of all towers are introduced in \S3.  Introducing a concept of a discrete degree map on a monoid, we introduce, in \S4, the generating  and the skew generating functions as formal Dirichlet series.  The inversion formula is formulated and proven in \S5. In \S6,  by a use of 
presentations of monoids, we give examples of monoids whose 
growth and skew-growth functions are not power series but are formal Dirichlet series (not necessarily convergent).


\section{Minimal common multiples}
We recall some basic concepts on monoids and fix notations.

\begin{definition}  1. A semigroup $M$ with the unit element 1 is called a {\it monoid}.

\smallskip
 2. A monoid $M$ is called {\it cancellative}  if a relation $aub=avb$ for elements $a,b,u,v\in M$ implies a relation $u=v$.

\smallskip
3. For two elements $u,v$ of $M$, we denote
\vspace{-0.3cm}
\[
u \mid_l v
\vspace{-0.2cm}
\]
if there exists an element $x\in M$ such that $v=ux$, and say that $u$ divides $v$ from the left, or, $v$ is a multiple of $u$ from the right.  

4. A monoid $M$ is called (left) {\it conical} if relations $u\mid_l v$ and $v\mid_l u$ for elements $u,v\in M$ imply $u=v$.
\end{definition}

For a  monoid $M$ which may not be conical, one can define an equivalence relation on $M$ by putting $u\sim v$ $\Leftrightarrow_{\text def.}$  $u\mid_l v$ and $v\mid_l u$.  The equivalence class of $u\!\in\!M$ may be denoted by $[u]$, but we shall often confuse   $u\!\in\! M$ with its class in $M/\!\!\sim$. What is important is that by this equivalence relation, the left divisibility relation are preserved (i.e. if $u\sim u'$, $v\sim v'$ and $u\mid_l v$ then $u' \mid_l v'$).  Thus, we obtain a partial order set structure "$\le$" (or $\ge$) on the quotient set $M/\!\sim$ induced by the left division relation $\mid_l$, i.e. $u\mid_l v \! \Leftrightarrow\! u\le v \! \Leftrightarrow\! v\ge u$.  We shall denote by $u\!<\!v$ (or $v\!>\!u$) if $u\mid_l v$ and $u\not\sim v$.  
 
 In case when $M$ is cancellative, the equivalence relation $\sim$ has much simpler interpretation as follows.
 \begin{ass}
 {\it 
 Let $M$ be a cancellative monoid. Then, the set of right invertible elements of $M$ coincides with the set of left invertible elements of $M$ and they form the largest subgroup, denote by $G$, of $M$. Then, $u\sim v$ for $u,v\in M$ if and only if $uG=vG$, that is:  
 }
\vspace{-0.1cm}
 \[
 M/\!\sim\ \ =\ M/G. 
\]
 \end{ass}
 \begin{proof}
 This is immediately verified from the definition. 
  \end{proof}
  
In general, the product structure on $M$ may not be preserved on the quotient $M/\!\!\sim$ (i.e.\  $u
\!\sim\! u'$ and $v\!\sim\! v'$ does not imply $uv\! \sim\! u'v'$ in general). If the product structure is preserved, then the quotient $M/\!\!\sim$ is automatically a conical monoid.

\begin{rem} 
In the present paper, the enumerations for the growth functions is done in the level of the quotient set $M/\!\!\sim$ but not of $M$, since we use only the division relations (i.e. the poset structure) on the quotient set  $M/\!\!\sim$ to calculate it. 
Thus the inversion function of the growth function can be formally calculated by  Moebius inversion formula for the poset (see [C-F]). 

However, \vspace{-0.1cm}
the poset  $M/\!\!\sim$ is not arbitrary but  admits natural left $M$-action
\[
M\ \times\  (M/\!\!\sim) \ \ \longrightarrow \  \  M/\!\!\sim .
\vspace{-0.1cm}
\]
This fact, later on, induces a quite important consequence that the inversion function can be enhanced  using this action (see "motivic" formulation of the inversion formula in \S5 {\bf Remark}), giving   more structures on the inversion function {\it due to the influence of the monoid structure on $M$}, typically Euler product formula for abelian Gaussian monoids (\S5 Ex.2).
\end{rem}

In the rest of the paper, we assume the following chain condition on $M/\!\!\sim$.

\bigskip
\noindent
{\bf Descending chain condition.} 
There does not exist an infinite strictly decreasing sequence 
$u_1>u_2>u_3> \cdots$ of elements  in $M/\!\!\sim$.

\bigskip

We consider two operations on the set of subsets of $M/\!\!\sim$: {\it common multiple set} and {\it minimal set}. For  a subset $J$  of $M/\!\!\sim$ (which may not necessarily be finite),  put 
\[
\begin{array}{cccll}
\mathrm{cm}(J)\!&:=\{\! &\! \! u\in M/\!\sim \!& \! \mid \quad j\mid_l u \quad \forall j\in J\ \} \\
\min(J)\!&:=\{\! &\!\!  u\in J & \! \mid \ \not\exists v\in J \  \text{ such that} \  v<u \ \},
\end{array}
\]
and their composition:  the {\it set of minimal common multiples } of the set $J$ by
\[
\mathrm{mcm}(J):=\min(\mathrm{cm}(J)).
\]
Actually, $\mathrm{cm}(J)$  may be the empty set. However, due to the {\bf Descending chain condition},  if $J\!\not=\!\emptyset$ then $\min(J)\!\not=\!\emptyset$. More precisely, we have the following fact.

\medskip
\begin{fact}
For any $u\in J$ there exists an element $v\in \min(J)$ such that $v\mid_l u$ holds.
\end{fact}


\section{Tower of  minimal common multiples}

Let $M$ be,  as in \S2, a monoid satisfying {\bf Descending chain condition}.  In this section, we introduce towers of  minimal common multiples
in $M/\!\!\sim$. 

\begin{definition}
A  {\it tower}  in $M$ of {\it height} $n\in\Z_{\ge0}$ is a sequence 
\[
T:=(I_0,J_1,J_2,\cdots,J_n) 
\]
of subsets of $M/\!\!\sim\setminus\{[1]\}$ satisfying the following:
\[
\begin{array}{rlllll}
\text{i) } & I_0\not=\emptyset \text{ and $I_0=\min(I_0)$} .  \\
\text{ii) } & I_k:=\mathrm{mcm}(J_k)\not=\emptyset   \text{ for }\  k=1,\cdots,n.\\
\text{ iii) } & J_k \subset I_{k-1} \ \ \text{such that}  \ \ 1\!<\!\#J_k  \text{ for }\  k=1,\cdots,n. 
\end{array}
\]
We call $I_0$ the {\it ground} of the tower $T$ and $I_k$  the {\it set of minimal common multiples on the $k$th stage} $J_k$ of the tower $T$. In particular, we denote  the set of minimal common multiples on the top stage by
\[
\qquad |T|:= I_n \  \  (=\mathrm{mcm}(J_n) \ \text{ if } n>0).
\]
The set of all towers in $M$  shall be denote by $\mathrm{Tmcm}(M)$. 
\end{definition}

\begin{definition} 
We put an {\it oriented graph structure on} $\mathrm{Tmcm}(M)$ as follows. 

i)  The set of vertices is equal to the set $\mathrm{Tmcm}(M)$ of all towers. 

ii)  An oriented edge from a tower $T$ of height $n$ to a tower $T'$ of height $n'$ 
is given if and only if $n'\!=\!n+1$ with 
$T=(I_0,J_1,\cdots,J_{n})$ and $T'=(I_0,J_1,\cdots,J_{n+1})$  (or, we shall write $T'=(T,J_{n+1})$ ). 
That is, $T'$ is a tower obtained by just adding one more stage above to the tower $T$.

We denoted again by $\mathrm{Tmcm}(M)$ the set equipped with this graph structure.
\end{definition}

Since any tower of height $n\ge1$ has exactly one immediate predecessor and any tower of height 0 has no predecessor,   each connected component of the graph $\mathrm{Tmcm}(M)$ is a rooted tree whose root is given by a tower of height $0$ of the form $T_0\!=\!(I_0)$ for a non-empty ground set $I_0\subset M/\!\!\sim \setminus\{[1]\}$ with $I_0\!=\!\min(I_0)$. 
Denoting the tree component by $\mathrm{Tmcm}(M,I_0)$, we have the decomposition:
\vspace{-0.2cm}
\[
\mathrm{Tmcm}(M) = \bigsqcup_{\substack{I_0\subset M/\!\!\sim \setminus\{[1]\}
\\ I_0=\min(I_0)} }  \mathrm{Tmcm}(M,I_0).
\]

\begin{example}
1.  Let $M$ be a free monoid. Then, any tower in $M$ is of height 0.
 
2. Let $M$ be a monoid, which admits least common right-multiples, that is, for any subset $J$ of $M$,
the set $\mathrm{mcm}(J)$ is either empty or consisting of a single element (in $M/\!\sim$).  Then, any tower has height at most 1.  ({\it Proof}.  For any tower $T$ of height $\ge1$, $I_1=\mathrm{mcm}(J_1)\not=\emptyset$ consists of a single element so that the Definition iii) of a tower prohibits to have $J_2$.)  

Thus, each component $\mathrm{Tmcm}(M,I_0)$ is star-shaped, consisting of the vertex $(I_0)$ (the ground) and the vertices of the form $(I_0,J_1)$ where 
$J_1$ is a finite subset of $I_0$ having more than two elements which have the common multiple.
\end{example}

\section{Generating functions $P_{M,I_0}(t)$ and  $N_{M,I_0}(t)$}

From present section, we fix a discrete degree map $\deg$ defined on a monoid $M$. Using it, 
we introduce a growth function $P_{M,I_0}(t)$ and a skew growth function $N_{M,I_0}(t)$  labeled by a set $I_0\subset M/\!\!\sim$ satisfying $I_0=\min(I_0)$. In particular, if the label set $I_0$ is the set  $\min\{M/\!\!\sim\setminus\{[1]\}\}$ of all minimal elements of $M$, we call them the growth function and the skew-growth function of $(M,\deg)$ and denote them by $P_{M,\deg}(t)$ and  $N_{M,\deg}(t)$,  respectively.


\begin{definition}
A {\it discrete degree map} on a monoid $M$ is a map
\[
\deg\  :\  M\quad  \longrightarrow \quad  \R_{\ge0}  \qquad 
\]
such that 
\  i)  \ $\deg(u)=0$ if and only if $u\sim 1$, 

\qquad \ \ \ ii) \ $\deg(uv)= \deg(u)+\deg(v)$ for any $u,v\in M$,

\qquad \ \ iii) \ $\#\{ u\in M/\!\!\sim\ \mid \ \deg(u)\le r\}<\infty$ for any $r\in\R_{>0}$.
\end{definition}
\noindent
If $u\!\mid_l\! v$ then ii) implies $\deg(u)\!\le\!\deg(v)$, and, hence, if $u\!\sim\! v$ then $\deg(u)\!=\!\deg(v)$ so that $\deg$ induces a poset map $M/\!\!\sim\to \R_{\ge0}$, denoted by the same notation "$\deg$". It is "strict" in the sense that $u\!<\!v$ implies $\deg(u)\!<\!\deg(v)$.
The  iii) implies that the range $\deg(M)$ is a discrete subset of $\R$, and
\[
d_{min}:=\inf\{ \ \deg(u) \ \mid\  u\in (M/\!\!\sim\setminus \{[1]\})\ \}
\]
is a positive constant (the {\it lowest charge} of the degree map). 
A monoid admitting a discrete degree map  automatically satisfies the descending chain condition.

\medskip
For any subset $A$ of $M/\!\!\sim$, put $\deg(A):=\inf\{\deg(u)\mid u\in A\}$.  In particular, we  call 

\vspace{-0.2cm}
\centerline{
$\deg(|T|):=\min\{\deg(\Delta)\mid \Delta\in I_n\}$ 
}

\vspace{0.2cm}
\noindent
the {\it minimal degree of the tower} $T$ of height $n$. 
\vspace{-0.1cm}

\begin{ass} {\it Let $T$ be  a tower of height $n$, then we have
}
\[
\deg(|T|)\ge (n+1)d_{min}.
\]
\end{ass}
\begin{proof} Put $T=(I_0,J_1,\cdots,J_n)$. 
Using the condition \S3 Def. iii) for a tower, we see that 
\[
\begin{array}{rllllll}
 \deg(|T|)&=     & \deg(I_n) &\ge & \deg(J_n)+d_{min} & (\S3\ \mathrm{Def. iii)})\\
&\ge& \deg(I_{n-1})+d_{min}& \ge &\deg(J_{n-1})+2d_{min}& (\S3\ \mathrm{Def.iii)})\\
&\ge &\cdots &\ge& \cdots & (\cdots) \\
&\ge& \deg(I_{1})+\!(n\!-\!1)d_{min}& \ge &\deg(J_{1}) +nd_{min}&  (\S3\ \mathrm{Def. i)})\\
&\ge& \deg(I_0)+n d_{min}& \ge&  (n+1)d_{min} & (4.2).
\end{array}
\vspace{-0.5cm}
\]
\end{proof}
\begin{rem}
i)  The existence of a discrete degree map implies automatically that the monoid  $M$ satisfies the descending chain condition.

ii) The existence of a discrete degree map implies automatically $\#(J_k)< \infty$  for any tower $T$ and $k\le$ height of $T$, since $\mathrm{mcm}(J_k)\not=\emptyset$ is assumed.
\end{rem}

Next, associated with any algebra $A$ (we shall use only the case $A\!=\!\Z$ in the present paper), let us introduce an algebra $R_A$ over $A$ equipped with a formal topology (see Remark 4.3 below). 
Set a topological $A$-module by
\[
R_A:=\Bigl\{\sum_{n=0}^{\infty} a_n t^{d_n} \Big| 
\substack{
\ a_n\in A \ (\forall n\in\Z_{\ge0}), \ \text{and}\  
\{d_n\}_{n\in\Z_{\ge0}} \text{ is}  \\ 
\text{ a sequence in $\R_{\ge0}$ divergent to $+\infty$.}
}
\Bigr\}
\]
where the system of (formal) neighbourhoods of $0\in R_A$ are given by 
\vspace{-0.2cm}
\[
\vspace{-0.2cm}
\qquad \qquad t^dR_{A}:= \Bigl\{\sum_{n=0}^{\infty} a_n t^{d_n}\in R_A \Big|  \min\{d_n\mid n\!\in\!\Z_{\ge0} \ \text{and} \ a_n\!\not=\!0\}\ge d\  \Bigr\}
\]
for $d\in\R_{\ge0}$. Then, the product w.r.t. this topology is well-defined by setting
\vspace{-0.2cm}
\[
\vspace{-0.2cm}
\sum_{d\in\R_{\ge0}}\big(\sum_{\substack{n,m\in\Z_{\ge0}\vspace{-0.1cm}\\d_n+e_m=d}} a_nb_m \big) t^{d}:=\big(\sum_{n=0}^{\infty} a_n t^{d_n}\big)\big(\sum_{m=0}^{\infty} b_m t^{e_m}\big) 
\vspace{-0.1cm}
\]
so that $R_A$ becomes an $A$-algebra.

\begin{rem}
In some literatures, from a topological view point, the ring $R_A$ is called the {\it Novikov ring}. However, for our later applications, it is convenient to regard it as the ring of {\it formal Dirichlet series} (see \cite{H-R}) as we explain below.

Let $f(t)=\sum_{n=0}^{\infty} a_n t^{d_n} \in R_A$ with $a_n\!\in\!\C$ ($n\!\in\!\Z_{\ge0}$). Then, by a change $t=\exp(-s)$ of the variable, we obtain a formal series 
\vspace{-0.3cm}
\[
f(\exp(-s))=\sum_{n=0}^{\infty} a_n \exp(-d_ns), 
\vspace{-0.3cm}
\] 
which is called a Dirichilet series (of exponential type) if $\underset{n\to\infty}{\lim}d_n\!=\!+\infty$.   
If the series converges (absolutely) at  $s_0\!\in\!\C$, then it also converges (absolutely) for all $\{s\in\!\C \mid\! \Re(s)\!>\!\Re(s_0)\}$ and defines a holomorphic function on that half plane.
The  product of  two series in  $R_\C$ is compatible with the product as holomorphic functions on their common absolutely convergent half plane. 
The holomorphic function may extends meromorphically to a branched covering region of $\C$. 
For our application ([S1,2,3,4]), we are interested in the locations of the poles and the zeros of these extended meromorphic functions (see \S5.5-7).
\end{rem}

We return to the construction of growth and skew-growth functions, which, due to the discreteness of the degree map,  belong to the ring $R_\Z$ of integral Dirichlet series. 

\bigskip
\noindent
{\bf 1.  Growth function $P_{M,\deg}(t)$ of $(M,\deg)$. }

\smallskip
For any subset $I_0$ of $M/\!\!\sim\setminus\{[1]\}$ with $I_0\!=\!\min(I_0)$, consider the submonoid  $M(I_0)$ of $M$ generated by $I_0$, that is, the smallest submonoid $N$ of $M$ satisfying 

i) an element of  $M$ whose equivalence class belongs to $I_0$ belongs to $N$, 

ii) if $u,v\in M$ satisfies $u\in N$ and $u\sim v$ then $v\in N$. 

\noindent
Then, in this new monoid $M(I_0)$, we can consider the division theory and define the equivalence relation as in \S2. Then, $M(I_0)/\!\!\sim$ is naturally embedded into $M/\!\!\sim$. Thus, the restriction of the degree map on $M$ to $M(I_0)$ induces also a degree map which we shall denote again by $\deg$. 

Then, the generating function of the degree map on $M(I_0)$, which is also called a {\it growth function} labeled by $I_0$, is defined as follows.

\[
P_{M,I_0}(t):=\sum_{u\in M(I_0)/\!\sim} t^{\deg(u)} =\sum_{d\in\R_{\ge0}} \#((M(I_0)/\!\sim)_d)\  t^{d}.
\]
Here, we put 

\centerline{
$
(M(I_0)/\!\sim)_d\!:=\!\{u\in M(I_0)/\!\sim \ \ \mid \ \deg(u)\!=\!d\}
$
}

\vspace{0.2cm}
\noindent
for any real number $d\!\in\!\R_{\ge0}$, which is a finite set due to the assumption on $\deg$, and therefore $P_{M,I_0}(t)\in R_{\Z}$.
  In particular, by choosing $I_0$ to be the minimal generating set  $I_0\!=\!\min(M/\!\!\sim\!\setminus\! \{1\})$ of $M$ 
  ($\Leftrightarrow M(I_0)\!=\!M$),  we define
 the {\it growth function of the monoid $M$ with respect to the degree map} $\deg$  by 
\[
P_{M,\deg}(t):=\sum_{u\in M/\!\sim} t^{\deg(u)} =\sum_{d\in\R_{\ge0}} \#((M/\!\sim)_d)\  t^{d}.
\]

\medskip
\noindent
{\bf 2.  Skew growth function $N_{M,\deg}(t)$ of $(M,\deg)$.}

\smallskip
For any tower $T$ of height $n$, consider the generating sum
\[
\sum_{\Delta\in |T|}
t^{\deg(\Delta)}.
\] 
It is well-defined in the ring $R_\Z$ for any tower $T$ due to the the finiteness iii) on the degree map, 
 and it belongs to the ideal $t^{\deg(|T|)}R_\Z $.

\begin{ass}
{\it 
To any non-empty set  $I_0\subset M/\!\sim\setminus\{[1]\}$ with $I_0=\min(I_0)$,  we define a {\it skew generating function} labeled by $I_0$ by
\[
N_{M,I_0}(t):= 1+ \sum_{T\in \mathrm{Tmcm}(M,I_0)}(-1)^{\#J_1+\cdots+\#J_{n}-n+1} \sum_{\Delta\in |T|}\ t^{\deg(\Delta)}
\]
where we recall that  $\mathrm{Tmcm}(M,I_0)$ is the connected component of the graph $\mathrm{Tmcm}(M)$ (consisting of all towers whose ground is the set $I_0$) and  the summation index $T$ is a tower of the form $(I_0,J_1,\cdots,J_{n})$ on the ground $I_0$.  
Then, the  formal sum  over the running index $T$ is convergent in the ring $R_\Z$.
}
\end{ass}
\begin{proof}
In order to show that the sum with respect to the running index $T\in \mathrm{Tmcm}(M,I_0)$ is convergent in the ring $R_\Z$, it is sufficient to show that for any positive number $r\in\R_{\ge0}$, the set $\{T\in \mathrm{Tmcm}(M,I_0)\mid \deg(|T|)\le r\}$ is finite.  

a) We first recall the inequality $\deg(|T|)\ge (n+1)d_{min}$ for any tower $T$ of height $n$. Thus the height of a tower whose minimal degree is bounded by a constant $r$ is bounded by $(r/d_{min})-1$. 

b) Next, let us show by induction on $n\in\Z_{\ge0}$ that the number of towers in $\#\{T\in \mathrm{Tmcm}(M,I_0)\mid \text{ the height of $T$ is equal to $n$ } \& \ \deg(|T|)\le r\}<\infty$. Since there is only one tower $T=(I_0)$ of height 0, the first induction hypothesis is satisfied.  Assume the result for $n$. Any tower $T'$ of height $n+1$ has the form $(T,J_{n+1})$ for the predecessor $T$ of height $n$. The requirement  $r\ge \deg(|T'|)$ implies the boundedness $r-d_{min} \ge \deg(|T|)$.  By induction hypothesis, the number of such $T$ is finite. Therefore, it is sufficient to show that for any tower $T\in \mathrm{Tmcm}(M,I_0)$ of height $n$, the number of its successors $T'$ such that $r\ge \deg(|T'|)$ is finite. The choice of $T'$ is determined by the choice of the subset $J_{n+1}$ of $I_n$, where  we have the equality: $\deg(|T'|)=\min\{\deg(\Delta)\mid \Delta\in |T'|=\mathrm{mcm}(J_{n+1})\} $. Since, for any $\Delta\in I_{n+1}:=\mathrm{mcm}(J_{n+1})$ and any $j\in J_{n+1}$, one has $j<\Delta$ ({\it proof.} That $j\le \Delta$ is obvious by $\Delta\in \mathrm{cm}(J_{n+1})$. But $\Delta\le j$ is impossible, if else, then any element  of $J_{n+1}$ is less of equal than $j$ which contradicts that $J_{n+1}$ consists of more than two elements (\S3 Def. ii)) so that $\deg(\Delta)\ge d_{min}+ \max\{\deg(j) \mid \ j\in J_{n+1}\}$, and, hence, $\deg(j)\le r-d_{min}$. This means that $J_{n+1}$ is a subset of $I_n':=\{j\in I_n\mid \deg(j)\le r-d_{min}\}$. However, by the discreteness condition (\S4 Def. iii)) on $\deg$, the number of elements of $I_n'$  is finite. This means the freedom of the choice of $J_n$ is also finite. 

Combining a) and b),  the proof of Assertion is completed.
\end{proof}

We define the {\it skew-growth function of} $(M,\deg)$ by  $I_0=\mathrm{min}(M/\!\!\sim\setminus\{[1]\}$:
\vspace{-0.1cm}
\[
N_{M,\deg}(t):=N_{M,\min(M/\!\!\sim\!\setminus\! \{1\})}(t).
\]

\section{Inversion formula for the growth function}

The main result of the present paper is formulated in the following theorem.

\begin{theo}
Let $M$ be a cancellative monoid equipped with a discrete degree map.  
Then we have the inversion formula in the ring $R_\Z$:
\[
\!\!\!\!\!\!\!\!\!\!\!\!\!\!(*) \qquad \qquad\qquad \qquad 
P_{M,\deg}(t)\cdot N_{M,\deg}(t) =1.\qquad \qquad \qquad \qquad 
\]
\end{theo}
\begin{proof}
For $d\in\R_{\ge0}$, put $m_d:=\#(\{u\in M/\!\sim\ \mid \deg(u)=d\} )$ so that  $P_{M,\deg}(t)=\sum_{d\in\R_{\ge0}} m_dt^d$.  Then, $(*)$ is equivalent to the  "infinite recursion" relation
\vspace{-0.2cm}
\[
\!\!\!\!\!\!\!(**) \qquad
m_d\ \ +\sum_{T\in \mathrm{Tmcm}(M,I_0)}(-1)^{\#J_1+\cdots+\#J_{n}-n+1} \sum_{\Delta\in |T|} m_{d-\deg(\Delta)}=0
\vspace{-0.2cm}
\]
for all $d\in\R_{> 0}$ due to the definition of the product  of formal Dirichlet series.

For any subset  $I$ and $J$ of $M/\!\sim$, let us introduce two sets:
\[\begin{array}{rlll}
\vspace{-0.05cm}
M^I &:=& \{u\in M/\!\sim\ \mid \exists \Delta\in I \text{ s.t. }  \Delta|_l u\}\\
M_J &:=& \{u\in M/\!\sim\ \mid \forall \Delta\in J \text{ s.t. }  \Delta|_l u\}.
\end{array}
\]\footnote
{The notation $M_J$ is confusing with  $M_d$ given in the previous section. However, we shall only use the suffixes $d$ and $J$ (or $J_1,J_2,\cdots$) so that they should be distinguished.
}
Note that we have an obvious relation:
\[
M_J= M^{\mathrm{mcm}(J)}  \quad \text{ and} \quad M^I=\cup_{J\in 2^I\setminus\{\emptyset\}} M_J
\]
for any subset $J$ and $I$ of $M/\!\sim$.  In particular, we have $M^\emptyset=\emptyset$ and $M_\emptyset=M/\!\sim$.

We also put  $M_d^I:=M_d\cap M^I$ and $M_{d,J}:=M_d\cap M_J$.

Since $M_d= M_d^{I_0}$ for $d\!>\!0$, we express $M_d=\cup_{\Delta\in I_0}M_{d,\{\Delta\}}=\cup_{J_1\in 2^{I_0}\setminus\{\emptyset\}} M_{d,J_1}$, where we have  $M_{d,J_1}=\cap_{\Delta\in J_1} M_{d,\{\Delta\}}$ for any non-empty subset $J_1 \subset I_0$,  and $M_{d,J_1}$ may be an empty set. We remark that $\{ J_1\in 2^{I_0}\mid M_{d,J_1}\not=\emptyset\}$ is a finite set of finite subsets of $I_0$,  since $M_{d,J_1}\!\not=\!\emptyset$ implies $J_1$ should be a subset of $\{\Delta\in I_0\mid \deg(\Delta)\le d\}$ which is finite due to the discreteness of the degree map.
Therefore, the following is a finite sum and has a meaning.
\[
\vspace{-0.1cm}
\#(M_d)=\sum_{J_1\in 2^{I_0}\setminus\{\emptyset\}}(-1)^{\#J_1-1}\#(M_{d,J_1}).
\vspace{-0.1cm}
\]

Since we have $M_{d,J_1}=M_d^{\mathrm{mcm}(J_1)}$, by putting $I_1:=\mathrm{mcm}(J_1)$, we have  $M_{d,J_1}=\cup_{\Delta\in I_1}M_{d,\{\Delta\}}$, and again decompose 
\vspace{-0.1cm}
\[
M_{d,J_1}=\cup_{J_2\in 2^{I_1}\setminus\{\emptyset\}} M_{d,J_1,J_2},
\vspace{-0.1cm}
\]
where $M_{d,J_1,J_2}\!:=\!M_{d,J_1}\!\cap\! M_{J_2} (=\!M_{d,J_2})$.
Even if $I_1$ may be infinite, we get again  a finite sum:
\[
\vspace{-0.1cm}
\#(M_{d,J_1})=\!\!\!\sum_{J_2\in 2^{I_1}\setminus\{\emptyset\}} (-1)^{\#J_2-1}\#(M_{d,J_1,J_2}).
\vspace{-0.1cm}
\]
Repeating the same process $n$-times, we obtain a formula
%
\[
\!\!\!\!\!(\bigstar) \ \#(M_d)=\!\!\!\!\!\!\sum_{J_1\in 2^{I_0}\!\setminus\!\{\emptyset\}}\!\sum_{J_2\in 2^{I_1}\!\setminus\!\{\emptyset\}}\!\!\!\!\cdots \!\!\!\!\!\sum_{J_n\in 2^{I_{n-1}}\!\setminus\!\{\emptyset\} }\!\!\!\!\!\!(-1)^{\#J_1\!+\!\#J_2\!+\!\cdots\!+\!\#J_n-n}\#(M_{d,J_1,J_2,\cdots,J_n})
\]
where we put $I_k=\mathrm{mcm}(J_k)$ ($k=1,\cdots,n-1$), and some of $M_{d,J_1,J_2,\cdots,J_n}=M_d\cap M_{J_1}\cap\cdots\cap M_{J_n} (=M_{d,J_n})$ may be an empty set.

\begin{ass}
{\bf i)}  {\it 
If  a running index $J_k$ of the formula $(\bigstar)$ for some $1\le k\le n$ consists only  of a single element, say $\Delta$, then  we have $J_k\!=\!J_{k+1}\!=\!\cdots\!=\!J_n\!=\!\{\Delta\}$ and $M_{d,J_1,\cdots,J_k}\!=\!M_{d,J_1,\cdots,J_k,J_{k+1}}\!=\!\cdots\!=\!M_{d,J_1,\cdots,J_k,\cdots,J_n}\!=\!\hat\Delta\cdot M_{d-\deg(\Delta)}$, where $\hat \Delta$ is any representative in $M$ of the equivalence class $\Delta\in M/\!\sim$.

{\bf ii)}  If $n\ge d/d_{min}$, then, for any running index $(J_1,\cdots,J_n)$ of the formula $(\bigstar)$, either the set $J_n$ consists of a single element, or the set $M_{d,J_1,J_2,\cdots,J_n}$ is empty.
}
\end{ass}
\begin{proof}
i) The fact $J_{k}=\{\Delta\}$ implies that $M_{d,J_1,\cdots,J_k}=\hat\Delta\cdot M_{d-\deg(\Delta)}$. On the other hand, we have $I_k:=\mathrm{mcm}(J_k)=\{\Delta\}$. Therefore, if $k<n$, then the only possible choice of $J_{k+1}$  is the only non-empty subset of $I_k$, i.e. $\{\Delta\}$. 

ii)  If $\#{J_n}\ge 2$. Then, due to above i), we should have all of $J_1,\cdots,J_n$ must have the cardinality greater or equal than 2. That is, for any $j\in J_{k+1}\subset \mathrm{mcm}(J_k)$, we have $\deg(j)\ge \max\{\deg(\Delta)\mid \Delta\in J_k\}+d_{min}\ge \deg(J_k)+d_{min}$ for $k=1,\cdots,n-1$ so that  we have $\deg(J_n)\ge \deg(J_{n-1})+d_{min}\ge \cdots\ge \deg(J_1)+(n-1)d_{min}\ge n d_{min}$. Therefore, any element of $M_{d,J_1,\cdots,J_n}=M_{d,J_n}$ (if it exists) has degree at least $\deg(J_n)\ge nd_{min}>d$, which is impossible. 
\end{proof}

\begin{ass}
{\it 
For any fixed $n\in\! \Z$ such that $n\ge d/d_{min}$,  we have a bijection
\[
\begin{array}{rll}
&& \{(\Delta,T)\in (M/\!\sim\times \mathrm{Tmcm}(I_0))\mid \Delta\in |T|, \deg(\Delta)\le d\}\\
&\simeq&
\{(J_1,\cdots,J_n)\mid J_1\!\subset\! I_0, \ J_k\!\subset\! \mathrm{mcm}(J_{k-1})\ (2\!\le\!k\!\le\! n),\ M_{d,J_1,\cdots,J_n}\not=\emptyset\}.
\end{array}
\]
If $(\Delta,T=(I_0,J_1,\cdots,J_{n_0}))$ and $(J_1,\cdots,J_n)$ are corresponding to each other, then we have the inequality $n_0<n$ and the equality
}
\[
(-1)^{\#J_1\!+\!\#J_2\!+\!\cdots\!+\!\#J_{n_0}-n_0}m_{d-\deg(\Delta)} = 
(-1)^{\#J_1\!+\!\#J_2\!+\!\cdots\!+\!\#J_n-n}\#(M_{d,J_1,J_2,\cdots,J_n}) .
\]
\end{ass}
\begin{proof}
a) Let $(\Delta,T)$ be an element in LHS, and let $T=(I_0,J_1,\cdots,J_{n_0})$. The condition implies $d\ge\deg(\Delta)\ge\deg(|T|)\ge (n_0+1)d_{min}$, and hence $n_0<d/d_{min}\le n$. Then $(J_1,\cdots,J_{n_0},\{\Delta\},\cdots,\{\Delta\})$ belongs to RHS.

b) Let $(J_1,\cdots,J_n)$ be an element in RHS such that $M_{d,J_1,\cdots,J_n}\not=\emptyset$. Due to Assertion 5.1 ii), in such case, we have $\#(J_n)=1$.  For such index, put $n_0+1=\inf\{1\le k\le n\mid \#J_k=1\}$ for some $0\le n_0<n$. Then $J_{n_0+1}=\{\Delta\}$ and $\Delta\in \mathrm{mcm}(J_{n_0})$ (or $\in I_0$ if $n_0=0$).  Since $\#(J_{n_0})\ge 2$ if $n_0\ge1$, $T:=(I_0,J_1,\cdots,J_{n_0})$ is a tower such that $(\Delta,T)$ belongs to LHS. 

It is clear that a) and b) are inverse to each other. Suppose $(J_1,\cdots,J_n)\leftrightarrow (\Delta,T)$, then we have the equality 
\[
\begin{array}{rlll}
& (-1)^{\#J_1\!+\!\#J_2\!+\!\cdots\!+\!\#J_n-n}\#(M_{d,J_1,J_2,\cdots,J_n})\\
=& (-1)^{\#J_1\!+\!\#J_2\!+\!\cdots\!+\!\#J_{n_0+1}-(n_0+1)}\#(M_{d,J_1,J_2,\cdots,J_{n_0+1}})\\
=& (-1)^{\#J_1\!+\!\#J_2\!+\!\cdots\!+\!\#J_{n_0}-n_0}\#(M_{d,J_{n_0+1}})\\
=& (-1)^{\#J_1\!+\!\#J_2\!+\!\cdots\!+\!\#J_{n_0}-n_0}\#(\hat\Delta\cdot M_{d-\deg(\Delta)}).
\end{array}
\]
 The cacellativity of the monoid $M$ implies the bijection $\hat\Delta\cdot M_{d-\deg(\Delta)}\!\simeq \!M_{d-\deg(\Delta)}$. Here, the bijection $\hat \Delta\cdot u\!\leftrightarrow\! u$ depends on a choice of $\hat \Delta$. However it does not effect on the enumeration of the cardinality of both hand sides, so that the last term is equal to 
 $
 (-1)^{\#J_1\!+\!\#J_2\!+\!\cdots\!+\!\#J_{n_0}-n_0}\#(M_{d-\deg(\Delta)}). 
$
\end{proof}

Last Assertion shows that the formula $(\bigstar)$ is the same as the recursion formula $(**)$.
This completes the proof of the recursion formula $(**)$, and the inversion formula $(*)$ of Theorem is proven.
\end{proof}
\begin{cor}
For any subset $I_0$ of $M/\!\!\sim\setminus\{[1]\}$ with $I_0=\min(I_0)$, we have
\vspace{-0.1cm}
\[
P_{M,I_0}(t) \cdot N_{M,I_0}(t)=1 .
\]
\end{cor}
\noindent
{\it Proof.} We have only to replace the monoid $M$ by its submonoid  $\langle I_0\rangle$. \qquad $\Box$

\begin{rem}  The inversion formula can be formulated in a "motivic" form:
\[
\vspace{-0.1cm}
\leftline{\!\!\!\!\! $(***)
\qquad\qquad   \qquad\qquad   \qquad\qquad      
\widehat{N}_{M}\ \cdot\ P_{M} \ \ =\ \  [1]$}
\]
where $P_{M}:=\sum_{u\in M/\!\sim}  u $ is an element of the module $\prod_{u\in M/\!\sim}\Z\cdot u$, on which
\vspace{-0.1cm}
\[
\widehat{N}_{M}:= 1+ \sum_{T\in \mathrm{Tmcm}(M,I_0)}(-1)^{\#J_1+\cdots+\#J_{n}-n+1} \sum_{\Delta\in |T|}\ \hat{\Delta} 
\vspace{-0.1cm}
\]
acts from left according to  {\bf Remark 2.2}.  The proof of  $(***)$ is a word for word translation of that of $(*)$ to "motivic" form.  Then the inversion formula $(*)$ is an immediate  consequence of $(***)$  by applying the homomorphism: 

\smallskip
\centerline{
$u\mapsto t^{\deg(u)}=e^{-\deg(u)s}$ \ \ and \ \ $\hat\Delta\mapsto t^{\deg(\Delta)}=e^{-\deg(\Delta)s}$.
}
\end{rem}

\begin{rem}
If the both Dirichlet series $P_{M,I_0}(\exp(-s))$ and $N_{M,I_0}(\exp(-s))$ 
converge absolutely in some half plane $\{s\!\in\!\C\mid \Re(s)\!>\!c\}$ for some $c\!\in\!\R$, then the inversion formula gives a functional equation $P_{M,I_0}(\exp(\!-\!s))N_{M,I_0}(\exp(\!-\!s))\!=\!1$ 
on the half plane. This implies, in particular,  that they do have neither zeros nor poles on the half plane. Let us denote by the same $P_{M,I_0}(\exp(\!-\!s))$ and $N_{M,I_0}(\exp(\!-\!s))$ their meromorphic continuations (including algebraic branches), respectively. Obviously, the functional equation extends meromorphically, i.e.\  one meromorphic continuation determines the other as the inverse function,  where the poles and the zeros of them interchange to each other.
Motivated by a trace formula for thermo-dynamical limit functions ([S1-4]),  we are interested in zeros of $N_{M,I_0}(\exp(-s))$. In particular, we ask the followings (c.f.\ \cite{T}).
\end{rem}

\begin{conj} 
If the Dirichlet series $P_{M,I_0}(\exp(-s))$ converges absolutely on some half plane $\{s\!\in\!\C\mid \Re(s)\!>\!\sigma\}$ for some $\sigma\!\in\!\R$, then $N_{M,I_0}(\exp(-s))$ also converges absolutely at least on the same half plane. 
\end{conj}

\begin{prob}
i)  Clarify the class of pair  $(M,\deg)$ of a cancellative monoid  with a discrete degree map on it, for which  following {\bf Assertion} holds. \footnote
{
This condition on the class of $(M,\deg)$ comes from a view point of the limit functions [S1]. One would like to think more stronger class of $(M,\deg)$, where  $N_{M,\deg}(\exp(-\!s))$ continues holomorphically on a open neighbourhood of the axis $\{\Re(s)\!=\!\sigma_0\}\!\subset\! \C$, where the order of zeros at $s\!=\!\sigma_0$ is the maximal among all zeros  on $\Re(s)=\sigma_0$.
}

  \smallskip
\noindent
{\bf Assertion.}\ {\it Let $\sigma_0\!\in\! \R$ be the minimal of $\sigma$'s in {\bf Conjecture 5.6}.\ Then  the Dirichlet series  $N_{M,\deg}(\exp(-\!s))$ continues holomorphically on a open neighborhood of the point $\sigma_0\!\in\!\!\C$ where the extended function takes value 0 at $\sigma_0$.\!}

\smallskip
\noindent
ii)   Determine the order of zeros of $N_{M,I_0}(\exp(-s))$ at $\sigma_0$ for monoids $(M,\deg)$ in the class in i),  and clarify its meaning for $(M,\deg)$.
\end{prob}

All the following Examples 1.\ i),ii),iii),iv),v), 2.\ and 3., satisfy  positively above Conjecture 5.6 and belong to the class in Problem 5.7.  In particular, it was conjectured in [S2] that Artin monoids belongs to the class in Problem 5.7 (in a stronger form) and was partially affirmatively solved in [K-T-Y].
In next \S6, we shall give a series of examples where Conjecture 5.6 is satisfied by all examples, but they may not always belong to the class in Problem 5.7.

\begin{example}
{\bf 1.}  For a monoid $M$ (cancellative and satisfying the descending chain condition) and a subset $I_0\!\subset\! (M/\!\!\sim\setminus\! \{[1]\})$  with $I_0\!=\!\min(I_0)$, set
\[
h(M, I_0):=\max\{ \text{height of } T \in \mathrm{Tmcm}(M,I_0)\}.
\]

\noindent
i) It is clear that $M$ is a free monoid if and only if $h(M,I_0)=0$ for any $I_0$.  

\noindent
ii) An Artin monoid (or, more generally, a monoid, any of whose finite subsets 

 admits the least common multiple) has $h(M,I_0)\le 1$ (see following 2.). 

\noindent
iii)  Ishibe \cite{I2} gave an example of $N_{M,\deg}(t)$  with $h(M,\deg)\!=\!2$.  

\noindent
iv) In general, if $h(M,I_0)<\infty$, then $N_{M,I_0}(t)$ is a polynomial in $t$. 

\noindent
v) Ishibe \cite{I2} has determined explicitly $N_{M,\deg}(t)$ for monoid  of type $\mathrm{B_{ii}}$  (\cite{I1},\cite{S-I}) 

and certain Zariski-van Kampen monoids, which have $h(M,\deg)\!=\!\infty$.

\medskip
\noindent
{\bf 2.}  Let $M\!:=\!\Z_{>0}$ with the ordinary product structure. 
We remark that the unit group of this case is a trivial group so that $\Z_{>0}/\!\!\sim\!=\!\Z_{>0}$.
As for the degree map, we take the logarithm function $\deg(n)\!:=\!\!\log(n)$ for $n\!\in\!\Z_{>0}$, which is  discrete since $\lim_{n\to\infty}\log(n)\!=\!\infty$. Then, by a change $t\!=\!\exp(-s)$ of variable, the growth function is  equal to Riemann's zeta-function (in the region $\Re(s)\!>\!1$)
\vspace{-0.2cm}
\[
P_{\Z_{>0},\log}(t)\ :=\ \sum_{n=1}^\infty  t^{\log(n)}=\ \sum_{n=1}^\infty  n^{\log(t)}=\ \sum_{n=1}^\infty n^{-s} \ =\ \zeta(s),
\vspace{-0.2cm}
\]
 which is well known to extend to the whole plane $\C$ with a simple pole at $s\!=\!1$.
 
 On the other hand, the skew-growth function for $(\Z_{>0},\log)$ is determined as follows. The ground set is given by $I_0\!:=\! \min(\Z_{>0}\!\setminus\!\{1\})\!=\!\{\text{prime numbers}\}$.  Since there exists always the least common multiple for any finite subset of $I_0$,\! all non-trivial towers  have height 1.\ Thus the skew-growth function is given by
\vspace{-0.1cm}
\[
N_{\Z_{>0},\log}(t)\ :=\!\!\!\! \sum_{\substack{J: \text{finite set of}\\ \text{prime numbers}}}\!\! \!\!\!(-1)^{\#J}\prod_{p\in J} t^{\log(p)} 
\  =\!\!\!\! \prod_{p: \text{prime numbers}}(1-p^{-s}).
\vspace{-0.2cm}
\]
Thus, the inversion formula turns out to be  the Euler product  formula:

\vspace{0.2cm}
\centerline{
$\zeta(s)\ \cdot\! \underset{p: \text{prime}}{\prod}(1-p^{-s})=1$
}
\noindent
for the Riemann zeta function.
Similar product formula holds for any abelian cancellative Gaussian monoid.

\medskip
\noindent
{\bf 3.}  We give an example whose towers can have arbitrarily large heights.

\smallskip
Consider the monoid $\!:=\!\langle a,b\mid a^2\!=\!b^2,\ ab\!=\!ba\rangle_{mo}$ (for a definition of this notation, see  \S6). The condition $\deg(a)\!=\!\deg(b)\!=\!1$ uniquely determine a degree map on $M$. Put $c\!:=\!a^2\!=\!b^2$. Then, any element of $M$ is uniquely expressed in the form
\vspace{-0.1cm}
\[
a^{\varepsilon_1}b^{\varepsilon_2} c^n   \quad \text{for some } \varepsilon_1,\varepsilon_2\in \{0,1\} , \ n\in\Z_{\ge0}.
\vspace{-0.1cm}
\]
Therefore, the growth function is given by
\vspace{-0.1cm}
\[
P_{W,\deg}(t) =\frac{1}{1-t^2}+\frac{t}{1-t^2}+\frac{t}{1-t^2}+\frac{t^2}{1-t^2}=\frac{1+t}{1-t}.
\vspace{-0.1cm}
\]
On the other hand, put $ J_n=\begin{cases}
 \{ac^{[n/2]},bc^{[n/2]}\} &\text{for odd $n\in\Z_{>0}$}\\
\{c^{[n/2]},abc^{[n/2]-1}\}  &\text{for even $n\in\Z_{>0}$.}
\end{cases}$
Then, one shows easily $\mathrm{mcm}(J_n)\!=\!J_{n+1}$ ($n\in\Z_{>0}$).  This implies that 
there exists a unique tower $T_n=(I_0,J_1,\cdots,J_n)$ of height $n\in\Z_{>0}$ with the ground set $I_0=\{a,b\}$.
Therefore, the skew growth function is given by
\vspace{-0.1cm}
\[
N_{W,\deg}(t)=1+\sum_{n: \ \text{odd}\ge1}^\infty 2t^{n+1}-\sum_{n: \ \text{even}\ge1}^\infty 2t^{n-1}=\frac{1-t}{1+t}.
\vspace{-0.1cm}
\]
\end{example}
\vspace{-0.1cm}

\section{\!\!Positive homogeneous presentation of a monoid}

In this section, we discuss  presentations of monoids by infinite generators and relations, which
are natural extension of the class of positive homogeneously presented monoids
studied in [S-I].
Using the presentation, we give examples of monoids whose growth and skew growth functions are not power series but  are  formal Dirichlet series which may not necessarily be convergent. 

Including the Zariski-van Kampen monoids of type $\mathrm{B_{ii}}$,  Ishibe has shown the cancellativity for several  monoids in this class and  calculated the skew-growth function $N_{M,\deg}(t)$ explicitly (see a forthcoming paper \cite{I2}).


\begin{definition}
1.  We call a pair 
$ \langle L \mid R\rangle_{mo}$
a presentation of a monoid defined below,  if  
i)  $L$ is a set, called generators,  and
ii) $R$ is a set, called relations, consisting of  expressions
 $R_i\!=\!S_i$  with  $R_i$ and $S_i$ are positive words $\in L^*$. 
Both sets  $L$ and $R$ may not necessarily be finite.

\medskip
2. The monoid  associated with the presentation $\langle L\mid R\rangle_{mo}$ is  defined as the quotient 
of the free monoid $L^*$ generated by $L$ by the equivalence relation $\simeq$
defined by the following 1) and 2). 

1) two words $U$ and $V$ in $L^*$ are called {\it elementarily 
equivalent} if either $U=V$ or $V$ is obtained from $U$ by substituting 
a substring $R_i$ of $U$ by $S_i$ where $R_i\!=\!S_i$ is a relation of $R$ 
($S_i=R_i$ is also a relation if $R_i=S_i$ is a relation), 

2) two words $U$ and $V$ in $L^*$ are called {\it equivalent}, denoted by 
$U\simeq V$  if there exists 
a sequence $U\!=\!W_0, W_1,\cdots, W_n\!=\!V $ of words in $L^*$ for $n\!\in\!\Z_{\ge0}$
such that $W_i$ is elementarily equivalent to $W_{i-1}$ for $i=1,\cdots,n$.

3) If  $U_1\simeq V_1$ and $U_2\simeq V_2$, then $U_1U_2\simeq V_1V_2$. Thus, we define the product between the equivalence classes.

\medskip
3. A map
$\deg:  L\longrightarrow \R_{\ge0},$
which naturally extends to a map $L^*\to \R_{\ge 0}$ additively denoted again by $\deg$, is called a degree map,
if i) $\deg(R_i)=\deg(S_i)$ for any relation $R_i=S_i$ in $R$, and ii) if $\deg(a)=0$ for some $a\in L$ then there exists $a'\in L^*$ and a relation $aa'=\emptyset$.  The condition i) implies that $\deg$ induces an additive map $\langle L\mid R\rangle_{mo}\to \R_{\ge0}$, denoted again by $\deg$ called a degree map. The condition ii) requires that $\deg(u)=0$ implies $u$ is invertible in the monoid. If, further, the inverse image $\deg^{-1}(0,r)$ of an interval $(0,r)\subset \R$ for any $r\in\R_{>0}$ intersect with $L$ by a finite set, then $\deg$ gives a discrete degree map on the monoid $\langle L\mid R\rangle_{mo}$.
\end{definition}

\begin{example}
 For any sequence ${\bf p}\!=\!\{p_k\}_{k\in\Z_{\ge0}}$ with $p_k\in\Z_{\ge0}$ ($k\in\Z_{\ge0}$) and $p_1\in 2\Z_{>0}$, let us consider a pair  $(M_{\bf p},\deg)$ presented as follows:
\[
\vspace{-0.2cm}
L=\{a_k\}_{k\in\Z_{\ge0}} \quad\text{and} \quad R=\{a_k^2=a_0^{p_k}a_{k-1}\}_{k\in\Z_{\ge1}}\cup\{a_ka_l=a_la_k\}_{k,l\in\Z_{\ge0}},
\]
\vspace{-0.1cm}
\quad $\deg: L\to \R_{>0},\quad a_k\mapsto d_k:=\frac{1}{2^k}+\sum_{i=1}^k\frac{p_i}{2^{k-i+1}},$

\medskip
\noindent
The degree map is discrete if the sequence $d_0\!=\!1,d_1,d_2,\cdots$ diverges to $+\!\infty$.

\begin{ass} 
{\it 
i) Any element $u$ of the monoid $M_{\bf p}$ has a unique expression:
\vspace{-0.1cm}
\[
\vspace{-0.1cm}
\!\!\!\!\!\!\!\!*) \qquad\qquad\qquad\qquad\qquad  a_0^n \prod_{k=1}^\infty  a_k^{\varepsilon_k}
\qquad\qquad\qquad \qquad\qquad\qquad 
\vspace{-0.1cm}
\]
for suitable $n\!\in\!\Z_{\ge0}$ and $\varepsilon\in E:=\{(\varepsilon_k)_{k\in\Z_{\ge1}}\in\! \{0,1\} ^{\Z_{\ge1}} \mid   \varepsilon_k\!=\!0 \text{ for } k\!>\!>\!0\}$.  
The degree map $\deg$ gives an embedding of $M_{\bf p}$ into $\R_{\ge0}$.

ii)  There exists an additive  function $m({\bf p},\delta)\in\Z$ on ${\bf p}$ depending  $\delta\in E$ such that  a number $r\!=\!m\!+\!\sum_{k=1}^\infty \delta_k/2^k\in \R$ for $m\!\in\!\Z$ and $\delta\in E$ belongs to the image $\deg(M_{\bf p})$ if and only if  $m\ge m({\bf p},\delta)$.
}
\end{ass}
\noindent
{\bf Terminology.}  We shall call the expression $*)$ the {\it normal form} of $u$, $a_0^n$ and $ \prod_{k=1}^\infty  a_k^{\varepsilon_k}$ the {\it integral part} and {\it 2-decimal part} of $u$, respectively.  
The number $\mathrm{depth}(u):=\max\{k\in\Z_{\ge1}\mid \varepsilon_k\not=0\}$ shall be called the {\it depth} of $u\in M_{\bf p}$.

\begin{proof}
i) Using the relations $R$, it is clear that any element $u$ of $\langle L,R\rangle_{mo}$ can be expressed in the form $*)$.
Then, we have 
\vspace{-0.25cm}
\[
\vspace{-0.2cm}
\deg(u) =n+ \sum_{k=1}^{\infty}\varepsilon_k \big(\frac{1}{2^k}+\sum_{i=1}^k\frac{p_i}{2^{k-i+1}} \big). 
\vspace{-0.05cm}
\]
Note that RHS is a finite sum w.r.t.\ the index $k$, and  expressed as a 2-decimal number. In view of the fact that the leading term (the deepest part) of the 2-decimal expansion of $d_k:=\frac{1}{2^k}+\sum_{i=1}^k\frac{p_i}{2^{k-i+1}}$ is equal to $\frac{1}{2^k}$ ($k\in\Z_{\ge0}$), the  2-decimal part of the number $\deg(u)\bmod 1$ determines the part $\varepsilon\!=\!(\varepsilon_k)_{k=1}^\infty$ of the word $*)$. So the exponent $n$ of the integral part $a_0^n$ of $*)$ is given by $\deg(u)\!-\!\sum_{k=1}^{\infty}\varepsilon_k d_k $. Thus the expression $*)$ is uniquely determined from $\deg(u)$, and  this implies that the correspondence $u\mapsto \deg(u)$ is injective.
 
 ii) For the given $r\!=\!m\!+\!\sum_{k=1}^\infty \delta_k/2^k\in \R$, we want to solve the equation
\vspace{-0.2cm}
 \[
 \vspace{-0.1cm}
n+ \sum_{k=1}^{\infty}\varepsilon_k \big(\frac{1}{2^k}+\sum_{i=1}^k\frac{p_i}{2^{k-i+1}} \big)=m+\sum_{k=1}^\infty \frac{\delta_k}{2^k}
\vspace{-0.1cm}
 \]
on $(n,\varepsilon)\!\in\!\Z_{\ge0}\!\times\! E$. The comparison of 2-decimal parts of both hand sides, says that $\varepsilon$ is uniquely determined from $\delta$ and {\bf p}. Then, in view of the equality: $n\!+\!\Big[\sum_{k=1}^{\infty}\varepsilon_k \big(\frac{1}{2^k}\!+\!\sum_{i=1}^k\frac{p_i}{2^{k-i+1}} \big)\Big]\!=\!m$ (here, "$[\ \cdot\ ]$" is the Gauss symbol),  the condition: $n\ge0$ is transformed to $m\!-\!\Big[\sum_{k=1}^{\infty}\varepsilon_k \big(\frac{1}{2^k}\!+\!\sum_{i=1}^k\frac{p_i}{2^{k-i+1}} \big)\Big]\!\ge\!0$. 
Put $m({\bf p},\delta):=\Big[\sum_{k=1}^{\infty}\varepsilon_k \big(\frac{1}{2^k}\!+\!\sum_{i=1}^k\frac{p_i}{2^{k-i+1}} \big)\Big]$ and rewriting $\varepsilon$ in RHS in terms of  $\delta$, we obtain the result.
\vspace{-0.1cm}
\end{proof}
\begin{cor} {\it Using the normal form $*)$, the growth function is given by 
\vspace{-0.05cm}
\[
\vspace{-0.05cm}
P_{M_{\bf p},\deg}(t)\ =\ \sum_{n=0}^\infty  t^n\cdot \prod_{k=1}^\infty \sum_{\varepsilon_k=0}^1 t^{\deg(a_k)\varepsilon_k}
\ =\  \frac{\prod_{k=1}^\infty (1+t^{\deg(a_k)})}{1-t}.
\vspace{-0.03cm}
\]
\vspace{-0.05cm}
\!\!\!The Dirichlet series $P_{M_{\bf p},\deg}(\exp(\!-\!s))$ (up\! to\! the\! factor\! $1\!-\!\exp(-s)$) converges on}
\end{cor}

\noindent
{\it the domain $\{s\!\in\!\C\mid \Re(s)\!>\!\!\overline{\underset{m\!\to\!\infty}{\lim}}\!\frac{1}{m}\!\log(\#\{n\!\in\!\Z_{\ge0} \mid [m]\!\le\! \!d_n\!\!<\!\!m\})\}$ (Kojima).}

\medskip
Applying this formula to the inversion formula $(*)$, we have the following description of the skew-growth function.
\[
N_{M_{\bf p},L}(t) =\frac{1-t}{\prod_{k=1}^\infty (1+t^{\deg(a_k)})}\quad \qquad \qquad \qquad \qquad \qquad 
\]
\[
\qquad =\sum_{\bf n} (-1)^{|{\bf n} |} \ t^{\deg(\Delta_{{\bf n} })}  + \sum_{\bf n} (-1)^{|{\bf n} |+1} \ t^{\deg(a_0\Delta_{{\bf n} } )} ,
\]
where we put 
$ \Delta_{\bf n}:=\prod_{k=1}^\infty a_k^{n_k} $
for any sequence ${\bf n}=\{n_k\}_{k=1}^{\infty}$ of non-negative integer where only finite $n_k$ are non-zero, and we put $|{\bf n}|=\sum_{k=1}^{\infty} n_k$.  

\begin{rem}
Let $r$ be the radius of absolute convergence of $\prod_{k=1}^\infty (1+t^{\deg(a_k)})$. Since $d_k=\deg(a_k)$ is a divergent sequence, one has $r\le 1$ and, hence,  the radius of absolute convergence of 
$P_{M_{\bf p},L}(t)$ is equal to $r$. On the other hand, it may be clear that $N_{M_{\bf p},L}(t)$ converges absolutely on the disc of radius $r$ so that Conjecture 5.6 is satisfied by these examples. However, whether $(M_{\bf p},\deg)$  belongs to the class in Problem 5.7. or not is an open question.
\end{rem}

We do not know how to obtain the expression $N_{M_{\bf p},\deg}(t)$  directly from the towers $\mathrm{Tmcm}(M_{\bf p},L)$.  In order to determine $\mathrm{Tmcm}(M_{\bf p})$,  we describe algorithms  i), ii) and  iii)  to obtain the  set 
$\mathrm{mcm}(J)$ for any finite subset $J$ of $M_{\bf p}$.  

\medskip
i) If the elements of $J$ have depth  bounded by $k\in\Z_{>0}$, then the depth of the elements of $\mathrm{mcm}(J)$ is also bounded by the same $k$.
\vspace{-0.1cm}
\begin{proof} 
Let  $v$ has depth $K\!>\!k$ and let $v\!=\!v'\prod_{j=k+1}^{K}a_j^{\varepsilon_j}$ be the normal form of  $v$ where $v'$ has depth at most $k$.  If $v$ is divisible by $u\in J$, then 
the decimal expansion $\deg(v/u)\!=\!\deg(v)\!-\!\deg(u)$ deeper than $k$ is unchanged from that of $v$. That is, $v/u\!=\!w\prod_{j=k+1}^{K}a_j^{\varepsilon_j}$ for some $w\!\in\! M_{\bf p}$.  That is, $v'/u\!=\!w\!\in\! M$. In other words, $v'\!=\!v/\prod_{j=k+1}^{K}a_j^{\varepsilon_j}$ is still divisible by $u\in J$, i.e.\  $v$ was not minimal.
\end{proof}

\vspace{-0.1cm}
ii) For each $\varepsilon\in E$, consider the set $M_\varepsilon:=\{a_0^n \prod_{k=1}^\infty  a_k^{\varepsilon_k}\mid n\in\Z_{\ge0}\}$. Then, there exists a unique minimal element of $M_\varepsilon\cap \mathrm{cm}(J)$, denoted by $\mathrm{mcm}_\varepsilon(J)$.
\vspace{-0.1cm}
\begin{proof}
The intersection $M_\varepsilon\cap \mathrm{cm}(J)$ is non-empty since $a_0^N$  is always divisible by a given finite set $J$ for sufficiently large $N$. Since $M_\varepsilon$ is inductively ordered by $\Z_\ge0$, the intersection  always has the minimal element. 
\end{proof}

\vspace{-0.1cm}
iii) In view of i) and ii),  $\mathrm{mcm}(J)\subset \{\mathrm{mcm}_\varepsilon(J)\mid \text{depth$(\varepsilon)\le$ depth$(J)$} \}$.
Since RHS is a finite set, we can choose its minimal subset $\mathrm{mcm}(J)$ by finite steps.

\medskip
Using the above  i)-iii), for any finite subset $J\!\subset\! L$ of the generator set $L$ and a positive integer $n\!\in\!\Z_{\ge0}$, we are algorithmically able to determine all towers having $J$ as its first stage  and height less or equal than $n$, and to check that minimal common multiples which appear in these towers have expressions of the form either $a_0\Delta_{\bf n}$ or $\Delta_{\bf n}$ for some $\bf n$  according as $a_0\!\in\! J$ or not (which are independent of the structure constant $\bf p$).  
However, an explicit list of all towers on the ground $L$ (i.e.\  the explicit description of $\mathrm{Tmcm}(M_{\bf p},L)$) is still unknown.


\end{example}

\begin{rem}
In [S-I], we have studied a particular subclass of presentation of monoids, where the set $L$ is finite and the set $R$ consists of relations of the form $R_i=S_i$ where $R_i$ and $S_i$ are words in $L$ of the same length. In this case, the map $\deg:L\to \{1\}$ automatically defines a discrete degree map on the monoid.  
Actually, Artin monoids [B-S] and many Zariski-van Kampen monoids (see e.g. [S-I], [I1],[I2]) belong to this class of monoids.  In a forthcoming paper [I2], skew growth functions for some monoids in this class, in particular for  the monoid of type $B_{ii}$ [I1], and also some other examples shall be determined explicitly.
\end{rem}

\end{document}